\documentclass[12pt,conference]{IEEEtran}
\ifCLASSOPTIONcompsoc
  \usepackage[nocompress]{cite}
\else
  \usepackage{cite}
\fi

\usepackage{amsmath,amssymb,url}
\usepackage{epic}
\usepackage{eepic}
\newsavebox{\wwide}
\newcommand{\wwidehat}[1]{\sbox{\wwide}{$#1$}
\ifdim\wd\wwide < 1.1 em \widehat{#1} \else
\setlength
{\unitlength}{0.01\wd\wwide}\overset
{\begin{picture}(100,6)
\path(0,0)(50,6)(100,0)
\end{picture}}{#1}\fi}

\newcommand{\dMp}{De Morgan poset}
\newcommand{\dMa}{De Morgan algebra} 
\newcommand{\wwidetilde}[1]{\sbox{\wwide}{$#1$}
\ifdim\wd\wwide < 1.1 em \widetilde{#1} \else
\setlength
{\unitlength}{0.01\wd\wwide}\overset
{\begin{picture}(100,6)
\path(0,0)(33,6)(45,6)(55,0)(67,0)(100,6)
\end{picture}}{#1}\fi}

\ifx\pdfoutput\@undefined\usepackage[usenames,dvips]{color}
\else\usepackage[usenames,dvipsnames]{color}
\IfFileExists{pdfcolmk.sty}{\usepackage{pdfcolmk}}{} 
\fi

\definecolor{ChadDarkBlue}{rgb}{.1,0,.2}  
\definecolor{ChadBlue}{rgb}{.1,.1,.5}  
\definecolor{ChadRoyal}{rgb}{.2,.2,.8}  
\definecolor{ChadGreen}{rgb}{0,.4,0}    
\definecolor{ChadRed}{rgb}{.5,0,.5}  
\usepackage{amsthm}
\usepackage{epic,eepic}
\usepackage[mathscr]{euscript}
\usepackage{enumerate}

\usepackage{diagrams}
\usepackage{graphicx}

\usepackage{lineno}
\usepackage{comment}

\usepackage{tikz}
\usetikzlibrary{graphs}
\usetikzlibrary{positioning}

\def\smallskip{\vskip\smallskipamount}
\def\medskip{\vskip\medskipamount}
\def\bigskip{\vskip\bigskipamount}

\numberwithin{equation}{section}

\theoremstyle{plain}
\newtheorem{theorem}{Theorem}[section]

\newtheorem{proposition}[theorem]{Proposition}
\newtheorem{corollary}[theorem]{Corollary}

\theoremstyle{definition}
\newtheorem{definition}[theorem]{Definition}

\newtheorem{remark}[theorem]{Remark}
\newtheorem{observation}[theorem]{Observation}

\newtheorem{example}[theorem]{Example}

\newcounter{ok}
{\end{list}}

\newcounter{aok}
{\end{list}}

\def\go#1;#2;#3 {\vbox to0pt{\kern-#3\hbox{\kern#2 #1}\vss}\nointerlineskip}

\begin{document}
\title{Set Representation of Partial Dynamic\\ \dMa{}s}
\author{\IEEEauthorblockN{Ivan~Chajda}
\IEEEauthorblockA{Department of Algebra and Geometry, Faculty of \\
                               Science,  Palack\'y University Olomouc, 17. listopadu\\
                             12,  CZ-771 46 Olomouc, Czech Republic\\
        E-mail: ivan.chajda@upol.cz}
\and 
\IEEEauthorblockN{Jan~Paseka}
\IEEEauthorblockA{Department of Mathematics and Statistics\\
		Faculty of Science, Masaryk University\\
		{Kotl\'a\v r{}sk\' a\ 2}, CZ-611~37~Brno, Czech Republic\\
        E-mail: paseka@math.muni.cz}}



\markboth{Set Representation of Dynamic \dMa{}s}%
{Set Representation of Dynamic \dMa{}s}
%

%




\IEEEcompsoctitleabstractindextext{%
\begin{abstract}
By a \dMa{} is meant a bounded poset equipped with an antitone 
involution considered as negation. Such an algebra can be 
considered as an algebraic axiomatization of a propositional 
logic satisfying the double negation law. Our aim is to 
introduce the so-called tense operators in every \dMa{} for 
to get an algebraic counterpart of a tense logic with negation 
satisfying the double negation law which need not be Boolean. 

Following the standard construction of tense operators $G$  
and $H$ by a  frame we solve the following question:  if 
a dynamic \dMa{} is given, how to find a  frame such 
that its tense operators $G$ and $H$ can be reached by this construction.

\end{abstract}

}

\maketitle  
\pagestyle{plain}

\IEEEdisplaynotcompsoctitleabstractindextext

%

{\section*{Introduction}}

\label{intro}

\IEEEPARstart{D}{ynamic \dMa{s} were already investigated by the authors in 
\cite{dem}. The reached theory is good enough for a description 
of tense operators in a logic satisfying the double negation law when 
a frame is given as well for the task to determine a frame provided 
tense operators are given.}

When studying partial dynamic \dMa{}s, we are given 
a De Morgan poset and we solve both the questions mentioned 
above. For this, we have to modify our original 
definition by axioms which are formulated in the language 
of ordered sets with involution only. On the other hand, 
we have an advantage of using the algebraic tools introduced 
already in \cite{dem} which can essentially shorten our paper.

For the reader convenience we repeat that tense operators are introduced 
for to incorporate the time dimension in the logic under consideration. 
It means that our given logic is enriched by the operators 
$G$ and $H$, see e.g. \cite{1} for the classical logic and  
\cite{chirita}, \cite{2}, \cite{chajda} for several non-classical logics.

It is worth noticing that the operators $G$ and $H$ can be considered as certain kind of modal  operators which were already studied 
for intuitionistic calculus by D. Wijesekera 
\cite{wijesekera} and in a  general setting by W.B.~Ewald \cite{Ewald}. For the logic of quantum mechanics (see e.g. \cite{dvurec} for details of the so-called 
quantum structures), the underlying algebraic structure is e.g. an 
orthomodular lattice or   the so-called effect algebra 
(see \cite{dvurec},\cite{FoBe}) 
and the corresponding tense logic was treated in 
\cite{chajdakolarik,dyn,dynpos,dyn2}, in a bit more general setting also in \cite{botur}.

The paper is organized as follows. After introducing 
several necessary algebraic concepts, we introduce tense operators in a 
\dMp, i.e., in an 
arbitrary logic satisfying double negation law without regards what another logical connectives are considered. Moreover, in this logic 
neither the principle of contradiction, nor the principle 
of excluded middle are valid for the negation, but all De Morgan laws hold.
Also we get a simple construction of tense operators which 
uses lattice theoretical properties of the underlying ordered set.  In Section
\ref{setreppartDMa}
we outline the problem of a representation of 
partial dynamic \dMa{}s and we 
solve it for partial dynamic \dMa{}s satisfying natural assumptions. 
This means that we get a procedure how to construct a corresponding  
frame to be in accordance with the construction from Section  
\ref{prelim}. In particular,  any dynamic \dMa{} 
is set representable.

\medskip

\section{{Preliminaries} and basic facts}
\label{prelim}

We refer the reader to \cite{Balbes} 
for standard definitions and notations for lattice structures.

\begin{definition}  A structure ${\mathbf A}=(A;\leq,',0,1)$ (${\mathbf A}=(A;\wedge,\vee,',0,1)$)  is called a
{\bfseries De Morgan poset} ({\bfseries De Morgan lattice}) if $(A;\leq,0,1)$ is a poset 
($(A;\wedge,\vee,0,1)$ is a lattice) with the 
top element $1$ and the bottom element $0$ 
and $'$ is a unary
operation called {\bfseries negation} with properties\ 
$a\leq b \Rightarrow b'\leq a'$\ 
and $a=a''$. 
\end{definition} 

In fact in a De Morgan poset $A$ we have $a\leq b$ iff  $b'\leq a'$,
because $a\leq b\Rightarrow b'\leq a'\Rightarrow a''\leq
b''\Rightarrow a\leq b$.

A  \textbf{ morphism} $f\colon A\to B$ \textbf{ of bounded  posets} 
(\dMp{}s)
is an order, (negation), top element and bottom element preserving map. 
A morphism $f\colon A\to B$ of bounded  posets 
is \textbf{order reflecting} if ($f(a)\leq f(b)$ if and only if $a\leq b$)  
for all $a, b\in A$.

Let $h\colon A \to B$ be a partial mapping of  \dMp{}s.  
We say that the partial mapping 
$h^{\partial}\colon A\to B$ is the {\bfseries dual of $h$} if  
$h^{\partial}(a)$ is defined 
for all $a\in A$ such that $a'\in \mathop{dom} h$ 
in which case
$$
h^{\partial}(a)=h(a')'.
$$

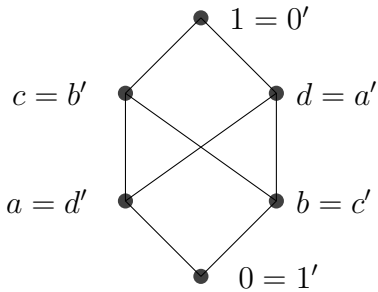
\begin{figure}[h]
\centering
\begin{tikzpicture}[scale=0.71654976583]
\coordinate [label=right:\phantom{ll}\hbox{$1=0'$}] (1) at (0,-1.2);
\coordinate [label=left:\phantom{lll}\hbox{$c=b'$}\phantom{lll}] (c) at (-1.4,-2.6);
\coordinate [label=left:\phantom{lll}\hbox{$a=d'$}\phantom{lll}] (a) at (-1.4,-4.6);
\coordinate [label=right:\phantom{l}\hbox{$d=a'$}\phantom{lll}] (d) at (1.4,-2.6);
\coordinate [label=right:\phantom{l}\hbox{$b=c'$}\phantom{lll}] (b) at (1.4,-4.6);
\coordinate [label=right:\phantom{lll}\hbox{$0=1'$}] (0) at (0,-6);
\draw (0) -- (a)  -- (c) -- (1);
\draw (0) -- (b)  -- (d) -- (1);
\draw (a)  -- (d);
\draw (b)  -- (c);
\fill[black,opacity=.75] (0) circle (4pt);
\fill[black,opacity=.75] (1) circle (4pt);
\fill[black,opacity=.75] (a) circle (4pt);
\fill[black,opacity=.75] (b) circle (4pt);
\fill[black,opacity=.75] (c) circle (4pt);
\fill[black,opacity=.75] (d) circle (4pt);
\end{tikzpicture}
\caption{The poset of Example \ref{benotbeJP}}\label{benzene}
\end{figure}

\begin{example}\label{benotbeJP} {\upshape{}The \dMp{} 
$\mathbf{M}=({M};\leq,$ $', 0,1)$, $M=\{0, a, b, c, d, 1\}$ displayed by the Hasse diagram in Figure \ref{benzene}
is the smallest non-lattice \dMp. }
\end{example}

\begin{observation}[\cite{dyn}]\label{obsik} Let $\mathbf A, \mathbf B$ be 
bounded posets (\dMp{s}), $T$ a set of morphisms from  $\mathbf A$ 
to $\mathbf B$ of bounded 
posets (\dMp{s}). 
The following conditions are equivalent:
\begin{enumerate}
\item[{\rm(i)}] \(((\forall t \in T)\, {t}(a)\leq {t}(b))\implies a\leq
b\) for any elements \(a,b\in A\);
\item[{\rm(ii)}] The morphism $i_{\mathbf A}^{T}\colon{}A \to B^{T}$ defined by 
$i_{\mathbf A}^{T}(a)=(t(a))_{t\in T}$ for 
all $a\in A$ is order reflecting.
\end{enumerate}
\end{observation}
We then say that $T$ is a {\bfseries full set of 
order preserving maps with respect to} ${\mathbf B}$. We may in this case 
identify ${\mathbf A}$ with a subposet (sub-\dMp{}) of ${\mathbf B}^{T}$ since 
$i_{\mathbf A}^{T}$ is an order reflecting 
morphism of bounded posets (\dMp{s}).

A pair 
$(f,g)$ of order-preserving mappings $f\colon{}A\to B$ and 
$g\colon{}B \to A$ between posets ${\mathbf A}$ and ${\mathbf B}$ 
is a {\bfseries  Galois connection} or 
an {\bfseries  adjunction} between ${\mathbf A}$ and ${\mathbf B}$ provided that 
$f(a)\leq b\ \text{if and only if}\ a\leq g(b)$ {for all}\ 
$a\in A, b\in B.$
In an adjunction $(f, g)$ the mapping $f$ is called  the {\bfseries  left adjoint} 
and the mapping $g$ is called the {\bfseries  right adjoint}. The pair 
$(f,g)$ of order-preserving mappings $f\colon{}A\to B$ and 
$g\colon{}B \to A$ is an adjunction if and only if 
$$a\leq g(f(a)) \, \text{and}\, f(g(b))\leq b\ \text{for all}\, 
a\in A, b\in B.$$

The second concept which will be used are so-called tense operators. They 
are in certain sense quantifiers which quantify 
over the time dimension of the logic under consideration. 
These tense operators were firstly introduced as operators on 
Boolean algebras (see \cite{1} for an overview). Chajda 
and Paseka introduced in \cite{dem} the notion of a  dynamic \dMa. 
\label{orto}

The following notion of a partial dynamic \dMa\ is stronger than the notion introduced in  \cite{dem}  but for 
 dynamic \dMa{s}\ both notions coincide.  Note only that our condition (P1) combined with the condition 
(T4) for tense De Morgan algebras in the sense of \cite{figallo} yields our condition (P4).

\begin{definition}\label{dadM}
  By a \textbf{partial dynamic \dMa{}} is meant a triple $\mathbf{D}=(\mathbf{A};G,H)$ such that 
$\mathbf{A}=(A;\leq,',0,1)$
  is a \dMp{} with negation $'$ and $G, H$ are partial mappings of $A$ into itself satisfying
  \begin{itemize}
    \item[(P1)] $G(0)=0$, $G(1)=1$, $H(0)=0$ and $H(1)=1$.
    \item[(P2)] $x\leq y$ implies $G(x)\leq G(y)$ whenever 
				$G(x), G(y)$ exist, and $H(x)\leq H(y)$ whenever 
				$H(x), H(y)$ exist.
    \item[(P3)] $x\leq GP(x)$ whenever $H(x')$ exists, $P(x)=H^{\partial}(x)$ and $GP(x)$ exists, 
and $x\leq HF(x)$ whenever $G(x')$ exists, $F(x)=G^{\partial}(a)$ and $HF(x)$ exists. 
 \item[(P4)]  $x\leq y$ implies $G(x)\leq F(y)$ whenever $G(y')$ and $G(x)$ exist,  and 
$H(x)\leq P(y)$ whenever $H(y')$ and $H(x)$ exist.
  \end{itemize}
  Just defined $G$ and $H$ will be called \textbf{tense operators} 
of a partial dynamic  \dMa{}  $\mathbf{D}$. If both $G$ and $H$ are total 
we will speak about a \textbf{dynamic   \dMa{}}. 

If we omit the condition (P3), i.e., only the conditions (P1), (P2) and (P4) 
are satisfied we say that $G$ and $H$ are \textbf{semi-tense operators on}   
$\mathbf{A}$.

If $(\mathbf{A}_1;G_1, H_1)$ and $(\mathbf{A}_2;G_2, H_2)$ are 
partial dynamic  algebras, then a {\bfseries morphism of 
partial dynamic  algebras} 
$f\colon (\mathbf{A}_1;G_1,  H_1)\to (\mathbf{A}_2;G_2,  H_2)$  is a morphism 
of \dMp{}s such that $f(G_1(a))=G_2(f(a))$, for any $a\in A_1$ such 
$G_1(a)$ is defined and $f(H_1(b))=H_2(f(b))$, for any $b\in A_1$ such 
$H_1(b)$ is defined.

Partial dynamic  \dMa{} $\mathbf{D}=(\mathbf{A};G,H)$ is called 
{\bfseries complete} if its reduct  $(A;\leq,$ $0,1)$ is a complete lattice.
\end{definition}

The semantical interpretation of these \textbf{tense operators} $G$ and $H$ is as follows.
Consider a pair $(T,\leq)$ where $T$ is a non-void set and $\leq$ is a partial order on $T$. Let $s\in T$ and $f(s)$ be a formula
of a given logical calculus. We say that $G\bigl( f(t)\bigr)$ \textbf{is valid} if for any $s\geq t$ the formula $f(s)$ is valid.
Analogously, $H\bigl( f(t)\bigr)$ is valid if $f(s)$ is valid for each $s\leq t$. Thus the unary operators $G$ and $H$ constitute an algebraic
counterpart of the tense operations ``it is always going to be the case that" and ``it has always been the case that", respectively. 
Similarly, the operators $F$ and $P$ can be considered in certain sense 
as existential quantifiers ``it will at some time be the case that" 
and ``it has at some time been the case that".

\medskip
 
In what follows we want to provide a meaningful procedure giving 
tense operators on every \dMp\ which will be in accordance with an 
intuitive idea of time dependency. 

By a \textbf{frame} (see e.g. \cite{2}) is meant a couple $(T,R)$ where $T$ is a non-void set 
and $R$ is a binary relation on $T$. Furthermore, we say that $R$  is {\bfseries serial} 
for all $x\in T$ there is $y\in T$ such that $x\mathrel{R}y$. 
In particular, every reflexive relation is serial. 
The set $T$ is considered to be a \textbf{time scale}, the relation $R$ expresses a relationship
``to be before" and  ``to be after". Having 
a \dMp{} $\mathbf{A}=(A;\leq,{}', 0, 1)$ and a non-void set $T$, 
we can produce the direct power
$\mathbf{A}^T=(A^T;\leq,{}', o,j)$ where the relation $\leq$ and the operation $'$ are 
defined and evaluated on $p,q\in A^T$ componentwise, i.e. $p\leq q$  if
$p(t)\leq q(t)$  for each $t\in T$ and $p'(t)=p(t)'$ for each $t\in T$.  
Moreover, $o, j$ are such elements of $A^T$ that $o(t)=0$ and $j(t)=1$
for all $t\in T$.

\begin{theorem}{\em\cite[Theorem II.7,Corollary II.9]{dem}} {}\label{upcomplate} 
Let $\mathbf{M}=(M;\leq,',0,1)$ be a complete 
\dMa{}  and let $(T,R)$ 
	be a frame. Define mappings $\widehat{G}, \widehat{H}$
  of $M^T$ into itself as follows: For all $p\in M^T$ and all $x\in T$, 
    $$\begin{array}{r c l}
\mbox{$\widehat{G}(p)(x)$}&=&\mbox{$\bigwedge_{M}\{p(y)\mid x \mathrel{R} y\} $}\phantom{.}\quad 
\text{and}\\ 
\mbox{$\widehat{H}(p)(x)$}&=&\mbox{$\bigwedge_{M}\{p(y)\mid y \mathrel{R}  x\} $}.
\end{array}$$
  Then 
\begin{enumerate}[{\rm (a)}]
\item $\widehat{G},\widehat{H}$  are total operators on 
$\mathbf{M}^T$.
\item If $R$ is serial then $\widehat{G}$  is a semi-tense operator.
\item If $R^{-1}$ is serial then $\widehat{H}$  is a semi-tense operator.
\item  If $R$ and  $R^{-1}$ are serial then $\mathbf{D}=(\mathbf{M}^T;\widehat{G},\widehat{H})$ is 
a dynamic  \dMa{}.  
\end{enumerate}  
\end{theorem}

We say that the {\bfseries operators $\widehat{G}$ and $\widehat{H}$
  on ${M}^T$ are constructed by means of $(T,R)$}.

\section{Set representation of partial dynamic 
De Morgan algebras}\label{setreppartDMa}

In Theorem \ref{upcomplate}, we presented a construction 
of natural tense operators when a 
\dMp{} and a frame are given. However, 
we can ask, for a given partial 
dynamic \dMa\ $(\mathbf{A};G,H)$, whether 
there exist a frame $(T,R)$ and a  complete de Morgan lattice  
$\mathbf{M}=(M;\leq,',0,1)$ such that the tense operators 
$G, H$ can be derived by this construction  
where $(\mathbf{A};G,H)$ is embedded 
into the power algebra $(\mathbf{M}^T;\widehat{G},\widehat{H})$.  
Hence, we ask, 
that there exists a suitable set $T$ and a binary relation 
$R$ on $T$ such that 
if every element $p$ of $A$ is in the form 
$(p(t))_{t\in T}$ in $M^T$ then   
$G(p)(s)= \bigwedge_M\{p(t)\mid s\mathrel{R}t\}$ for all $p\in \mathop{dom} G$ 
and $s\in T$, 
and $H(p)(s) = \bigwedge_M\{p(t)\mid t\mathrel{R}s\}$  for all 
$p\in \mathop{dom} H$ and $s\in T$. 
If such a representation exists then one can recognize 
the time variability of elements of $A$ expressed 
as time dependent functions $p\colon{}T \to M$ and 
$(\mathbf{A};G,H)$ is said to be {\bfseries representable in} 
 $\mathbf{M}$  {\bfseries with respect to} $T$.

In what follows, we will show that there is  a set representation 
theorem for  (partial) dynamic De Morgan algebras.  

Let us start with the following example.

\begin{example} \label{TDDMAboolupcomplate} \upshape Let 
$\mathbf{2}=(\{0, 1\};\vee, \wedge, ', 0, 1)$ be a 
two-element Boolean algebra. We will denote by 
$\mathbf{M}_{2}=({M}_{2};\leq, ', 0, 1)$ 
a complete De Morgan lattice such that  
${M}_{2}=\{0, 1\}\times \{0, 1\}$, 
$({M}_{2};\vee, \wedge,  0, 1)$ is a lattice reduct of  the Boolean  
algebra $\mathbf{2}\times \mathbf{2}$ with the induced order $\leq$ and the negation 
on ${M}_{2}$ is defined by  $(a,b)'=(b', a')$.
Let $(T,R)$ 	be a  frame. 
Let the operators $\widehat{G}$ and $\widehat{H}$
  on ${\mathbf M}_{2}^T$ be constructed by means of $(T,R)$. 
  Then by Theorem \ref{upcomplate}   
  we have that 
$(\mathbf{M}_2^T;\widehat{G},\widehat{H})$  is 
a complete dynamic De Morgan algebra. 
Moreover, ${\mathbf M}_{2}^T$ is isomorphic as a lattice to  a De Morgan algebra of sets.
\end{example}

 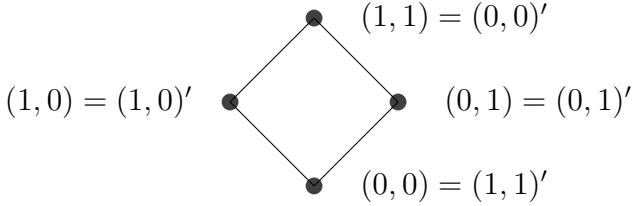
\begin{figure}[h]
\centering
\begin{tikzpicture}[scale=0.7976583]
\coordinate [label=left:\phantom{lll}\hbox{$(1,0)=(1,0)'$}\phantom{lll}] (a) at (-1.4,-4.6);
\coordinate [label=right:\phantom{llll}\hbox{$(1,1)=(0,0)'$}\phantom{lll}] (e) at (0,-3.2);
\coordinate [label=right:\phantom{llll}\hbox{$(0,1)=(0,1)'$}\phantom{lll}] (b) at (1.4,-4.6);
\coordinate [label=right:\phantom{llll}\hbox{$(0,0)=(1,1)'$}] (0) at (0,-6);
\draw (0) -- (a); 
\draw (0) -- (b) ; 
\draw (a)  --  (e);
\draw (b)  --  (e);
\fill[black,opacity=.75] (0) circle (4pt);
\fill[black,opacity=.75] (a) circle (4pt);
\fill[black,opacity=.75] (b) circle (4pt);
\fill[black,opacity=.75] (e) circle (4pt);
\end{tikzpicture}
\caption{Figure of the underlying poset of the complete De Morgan lattice  $\mathbf{M}_{2}$ 
from Example \ref{TDDMAboolupcomplate}}\label{FigM2v}
\end{figure}

Recall that  the four-element  De Morgan poset $\mathbf{M}_{2}$, considered as a distributive De Morgan lattice, 
generates the variety of all distributive De Morgan lattices (see e.g. \cite{Balbes}, \cite{Petrovich}).  
This result was a motivation for our study of the representation theorem of  De Morgan posets.

For any  De Morgan poset 
 ${\mathbf A}=(A;\leq, ', 0, 1)$, we will denote by 
  $T_{\mathbf A}^{\text{DMP}}$ a set of morphisms of  
De Morgan poset into the four-element  De Morgan 
poset $\mathbf{M}_{2}$. 
The elements $\kappa_D\colon{} A\to M_2$ of 
$T_{\mathbf A}^{\text{DMP}}$ (indexed by 
proper down-sets $D$ of  $\mathbf{A}$ which correspond to morphisms 
of bounded posets $h_D\colon{} A\to \{0, 1\}$ 
such that $h_D(a)=0$ iff $a\in D$) are 
morphisms of  De Morgan  posets defined by the prescription 
${{\kappa}_{D}}(a)=({h_{D}}(a), {h_{D}}^{\partial}(a))$ for all $a\in A$. 

As we will see later, this set $T_{\mathbf A}^{\text{DMP}}$ will serve as our time scale.  
Hence, our next task is to determine a binary relation $R$ on $T_{\mathbf A}^{\text{DMP}}$ 
such that the couple $(T_{\mathbf A}^{\text{DMP}}, R)$ will be our appropriate frame.

Let us denote, for any proper down-set $D$ of  $\mathbf{A}$, by 
$\partial(D)$ the set $A\setminus \{d'\mid d\in D$\}. 
Then $\partial(D)$ is again a 
 proper down-set of  $\mathbf{A}$ such that 
$ {h_{D}}^{\partial}={h_{{\partial}(D)}}$ and we have 
$ {h_{D}}={{h_{{\partial}(D)}}^{\partial}}$.

To simplify the notation we will use for elements of 
$T_{\mathbf A}^{\text{DMP}}$ 
letters $s$ and $t$ whenever we will need not their 
concrete representation via down-sets.


\begin{proposition}\label{distDMPvloyeni} Let ${\mathbf A}=(A;\leq, ', 0, 1)$ 
be a  De Morgan  poset{}. Then the map 
$i_{{}{\mathbf A}}\colon{}A\to M_2^{T_{\mathbf A}^{\text{DMP}}}$ given 
by $i_{{}{\mathbf  A}}(a)(s)=s(a)$ for all $a\in A$ and all 
$s\in T_{\mathbf A}^{\text{DMP}}$ 
is an order-reflecting morphism of De Morgan posets 
such that $i_{{}{\mathbf  A}}(A)$ is 
a De Morgan subposet of \/ ${\mathbf  M}_2^{T_{\mathbf A}^{\text{DMP}}}$.
\end{proposition}
\begin{proof} First, let us show that, for any proper   down-sets $D$ of  $\mathbf{A}$, 
the mapping $\kappa_D\colon{} A\to M_2$ is a morphism of De Morgan posets. Since 
both ${h_{D}}$ and ${h_{D}}^{\partial}$ are order-preserving we have that 
$\kappa_D$ is order-preserving. Now, let us compute 
$$\begin{array}{l}
{h_{D}}^{\partial}(0)=({h_{D}}(0'))'=({h_{D}}(1))'=1'=0,\\
{h_{D}}^{\partial}(1)=({h_{D}}(1'))'=({h_{D}}(0))'=0'=1.
\end{array}
$$
It follows that both ${h_{D}}$ and ${h_{D}}^{\partial}$ preserve $0$ and $1$, i.e., 
they are morphisms of bounded posets. Hence $\kappa_D$ is 
a morphism of bounded posets. 

Let $a\in A$. Let us check that ${{\kappa}_{D}}(a')={{\kappa}_{D}}(a)'$. 
We have 
$$
\begin{array}{r c l}
{{\kappa}_{D}}(a)'&=&({h_{D}}(a), {h_{D}}^{\partial}(a))'\\
&=&%
({h_{D}}^{\partial}(a)', {h_{D}}(a)')= %
({h_{D}}(a'), {h_{D}}(a'')')\\
&=&({h_{D}}(a'), {h_{D}}^{\partial}(a'))=%
{{\kappa}_{D}}(a').
\end{array}
$$
This yields that $\kappa_D$ is a morphism of De Morgan posets. 

It remains to check that $T_{\mathbf A}^{\text{DMP}}$ a full set of morphisms. 
Let $a, b\in A$ such that $s(a)\leq s(b)$ for all $s\in T_{\mathbf A}^{\text{DMP}}$.
In particular, ${h_{D}}(a)\leq {h_{D}}(b)$ for all proper   down-sets $D$ of  $\mathbf{A}$. Put $D=\{x\in A\mid x\leq b\}$. 
Since  $ {h_{D}}(b)=0$ we have $ {h_{D}}(a)=0$, i.e., 
$a\in D$ and hence $a\leq b$.
\end{proof}

The next theorem solves our problem of finding 
the binary relation $R$ on  $T_{\mathbf A}^{\text{DMP}}$. 
In fact, we are restricted here on a semi-tense operator 
$G$ only.

\begin{theorem}\label{semiDMPdreprest} 
 Let   ${\mathbf A}=(A;\leq, ', 0, 1)$  be  
a De Morgan poset, 
 $G\colon{}A\to A$ a semi-tense operator on ${\mathbf A}$. 
 Let us put 
 $$
 \begin{array}{r@{\,}c@{\,}l}
 R_{G}&=&\{(s, t)\in %
 T_{{\mathbf A}}^{\text{DMP}}\times T_{{\mathbf A}}^{\text{DMP}}\mid 
(\forall x\in \mathop{dom}(G))\\
& &\phantom{\{(s, t)\in  T_{{\mathbf A}}^{\text{DMP}}\times T_{{\mathbf A}}^{\text{DMP}}\mid\ }%
(s(G(x))\leq t(x))\}.\\  
\end{array}$$

Then $(T_{{\mathbf A}}^{\text{DMP}},R_{G})$ is 
a  frame with $R$ being serial. 
Let 
$\widehat{G}$  be 
the operator constructed by means 
of the  frame $(T_{{\mathbf A}}^{\text{DMP}},R_{G})$ 
and let us put $\widehat{F}={\widehat{G}}^{\partial}$.
Then the mapping  $i_{{}{{\mathbf A}}}$ is 
an order-reflecting morphism of De Morgan posets 
into the complete De Morgan lattice 
${\mathbf  M}_2^{T_{{\mathbf A}}^{\text{DMP}}}$ 
such that the following diagram commutes:
$$
\begin{diagram}
{A}&%
\lTo(2,0)^{{}{F}}&{A}&%
\rTo(2,0)^{{}{G}}&{A}&&\\
\dTo(0,3)^{i_{{}{{\mathbf  A}}}}&&\dTo(0,3)^{i_{{}{{\mathbf  A}}}}&&%
\dTo(0,3)_{i_{{}{{\mathbf  A}}}}&&\\
{\mathbf  M}_2^{T_{{\mathbf A}}^{\text{DMP}}}&\lTo(2,0)_{\widehat{F}}&
{\mathbf  M}_2^{T_{{\mathbf A}}^{\text{DMP}}}&\rTo(2,0)_{\widehat{G}}&%
{\mathbf  M}_2^{T_{{\mathbf A}}^{\text{DMP}}}&&
\end{diagram}
$$
\end{theorem}
\begin{proof} First, let us  verify that 
the following holds: 
\begin{enumerate}
\item  for all $b\in \mathop{dom}(G)$ and 
for all $s\in T_{{\mathbf A}}^{\text{DMP}}$, 
$s(G(b))=\bigwedge_{M_2}\{t(b)\mid  s \mathrel{R_{G}} t\}$,
\item  for all $b\in  \mathop{dom}(F)$ 
and for all $s\in T_{{\mathbf A}}^{\text{DMP}}$, 
$s(F(b))=\bigvee_{M_2}\{s(b)\mid  s \mathrel{R_{G}} t\}$.
\end{enumerate}

Let us first check statement 1. 

Assume that $b\in \mathop{dom}(G)$ 
and $s\in T_{{\mathbf A}}^{\text{DMP}}$,  
$s=\kappa_D$  where $D$ is a proper   down-set of  $\mathbf{A}$. 
Then, for all $t\in T_{{\mathbf B}}^{\text{DMP}}$ such that  
$s \mathrel{R_G} t$, $s(G(b)) \leq t(b)$. Hence 
$$(h_D(G(b)), h_D^{\partial}(G(b)))%
\leq \mbox{$\bigwedge_{M_2}$}\{t(b)\mid  s \mathrel{R_G} t\}.$$ 
To get the other 
inequality assume that 
$h_D(G(b))=0$ or  $h_D^{\partial}(G(b))=0$. 

Assume that $h_D(G(b))=0$.
Put $V=\{z\in A\mid$ $(\exists x\in \mathop{dom(G)})(z\geq x\ 
\text{and}\ h_D(G(x))=1)\}$ and 
$X=\{z\in A\mid (\exists y\in \mathop{dom(F)})(z\leq y$   
$\text{and}\ h_D(F(y))=0)\ \text{or}\ z\leq b\}$. Then 
$V$ is a proper upper subset  of ${\mathbf A}$, $1\in V$, and 
$X$ is a proper downset of ${\mathbf A}$, $0\in X$ such 
that $X\cap V=\emptyset$ and $b\in X$. To verify this, 
assume that there is an element 
$z\in X\cap V$. Then there is $x\in \mathop{dom(G)}$, $x\leq z$ 
such that $h_D(G(x))=1$.    
Also, we have that 
there is $y\in \mathop{dom(F)}$, $z\leq y$ 
such that $ h_D(F(y))=0$ or $x\leq z\leq b$). It follows 
that $x\leq y$ and  $1=h_D(G(x))\leq h_D(F(y))=0$ 
or $h_D(G(b))=1$, 
a contradiction.

Let $U$ be a maximal down-set of ${\mathbf A}$ 
including $X$ such that  $V\cap U=\emptyset$.
Hence $U$ determines 
a morphism  $h_U\colon{}A\to \{ 0,1\}$  
of bounded posets 
such that ${h_{U}}(z)=0$  for all 
$z\in X$ and ${h_{U}}(z)=1$ for all 
$z\in V$, i.e.,  $h_D(G(x))\leq {h_{U}}(x)$ for all 
$x\in \mathop{dom(G)}$. 
Let us check that 
${h_{D}}^{\partial}(G(x))\leq {h_{U}}^{\partial}(x)$ for 
all $x\in \mathop{dom(G)}$.  
Assume that $1={h_{D}}^{\partial}(G(x))=h_D(F(x'))'$. Then 
$h_D(F(x'))=0$, i.e., $x'\in X$. It follows that ${h_{U}}(x')=0$, i.e., 
${h_{U}}^{\partial}(x)=1$.
But this yields that 
$(\forall x\in \mathop{dom}(G)) %
\phantom{\{}(s(G(x))\leq \kappa_U(x))$, i.e., 
$s \mathrel{R_G}  \kappa_U=(h_U(-),  {h_{U}}^{\partial}(-))$
and ${h_{U}}(b)=0$.

Assume now that 
$h_D^{\partial}(G(b))=h_{{\partial}(D)}(G(b))=0$.  
As above, there is a maximal downset $W$ of ${\mathbf A}$ such that 
${h_{\partial(W)}}^{\partial}(b)= {h_{W}}(b)=0$, 
$$h_{{\partial}(D)}(G(x))\leq {h_{W}}(x)\ \text{and }\ 
h_{{\partial}(D)}^{\partial}(G(x))\leq {h_{W}}^{\partial}(x)$$ 
 for all $x\in \mathop{dom}(G)$. It follows that 
$h_D(G(x))\leq {h_{\partial(W)}}(x)$ and 
${h_{D}}^{\partial}(G(x))\leq {h_{\partial(W)}}^{\partial}(x)$ for all $x\in B$, 
i.e.,  $s \mathrel{R_{G}} \kappa_{\partial(W)}$ and 
${h_{\partial(W)}}^{\partial}(b)=0$.

Consequently, $s(G(b))= \bigwedge_{M_2}\{t(b)\mid  s \mathrel{R_G} t\}$.

Let us check Statement 2. We have 
$$
\begin{array}{@{}r@{}l}
s(F(&b))=s(G(b')')=s(G(b'))'\\[0.15cm]%
&=(\bigwedge_{M_2}\{t(b')\mid  s \mathrel{R_{G}} t\})'\\[0.15cm]%
&=\bigvee_{M_2}\{t(b')'\mid  s \mathrel{R_{G}} t\}%
=\bigvee_{M_2}\{t(b)\mid  s \mathrel{R_{G}} t\}.
\end{array}
$$

It remains to verify that $R_G$ is serial. 
Let $s\in  T_{{\mathbf A}}^{\text{DMP}}$. We know from $(P1)$ that 
$0=s(G(0))= \bigwedge_{M_2}\{t(0)\mid  s \mathrel{R_G} t\}$. 
The set $\{t\in T_{{\mathbf A}}^{\text{DMP}}\mid  s \mathrel{R_G} t\}$ 
is non-empty (otherwise one has $0=s(G(0))=1$, a contradiction). 
\end{proof}

In what follows, we show that if $G$ and $H$ are semi-tense operators 
such that the induced relations $R_G$ and $R_H$ satisfy a natural 
condition $R_G=(R_H)^{-1}$ then the obtained frame is just the one  
we asked for.

\begin{theorem}\label{seDMPdreprest} 
 Let   ${\mathbf A}=(A;\leq, ', 0, 1)$  be  
a De Morgan poset, 
 $G, H\colon{}A\to A$ be  semi-tense operators on ${\mathbf A}$ 
such that $R_G=(R_H)^{-1}$.

Then $(T_{{\mathbf A}}^{\text{DMP}},R_{G})$ is a  frame with $R_G$ 
and $ (R_G)^{-1}$ serial. 
Let 
$(\mathbf{M}_2^{T_{{\mathbf A}}^{\text{DMP}}};\widehat{G},\widehat{H})$
be the dynamic De Morgan algebra constructed by means 
of the  frame $(T_{{\mathbf A}}^{\text{DMP}},R_{G})$.
Then the mapping  $i_{{}{{\mathbf A}}}$ is 
an order-reflecting morphism of De Morgan posets 
into the complete De Morgan lattice 
${\mathbf  M}_2^{T_{{\mathbf A}}^{\text{DMP}}}$ 
such that the following diagram commutes:
$$
\begin{diagram}
{A}&%
\lTo(2,0)^{{}{H}}&{A}&%
\rTo(2,0)^{{}{G}}&{A}&&\\
\dTo(0,3)^{i_{{}{{\mathbf  A}}}}&&\dTo(0,3)^{i_{{}{{\mathbf  A}}}}&&%
\dTo(0,3)_{i_{{}{{\mathbf  A}}}}&&\\
{\mathbf  M}_2^{T_{{\mathbf A}}^{\text{DMP}}}&\lTo(2,0)_{\widehat{H}}&
{\mathbf  M}_2^{T_{{\mathbf A}}^{\text{DMP}}}&\rTo(2,0)_{\widehat{G}}&%
{\mathbf  M}_2^{T_{{\mathbf A}}^{\text{DMP}}}&&
\end{diagram}
$$
\end{theorem}
\begin{proof} It immediately follows from Theorem \ref{upcomplate}  and Theorem 
\ref{semiDMPdreprest}.
\end{proof}

The following theorem gives us a complete solution 
of our problem established in the beginning of this section. This is a new result 
showing that the partial dynamic \dMa{} can be equipped 
with the corresponding frame similarly as it is 
known for Boolean algebras in \cite{1} at least in a case 
when the tense operators $G$ and $H$ are interrelated. 
In particular, any dynamic  \dMa{} has such a frame.

\begin{theorem}\label{sercDMPdreprest} 
 Let   $({\mathbf A};G,H)$  be  
a partial dynamic De Morgan algebra such that 
\begin{enumerate}[(a) ] 
\item $x\in \mathop{dom}(G)$ implies $G(x)'\in \mathop{dom}(H)$, 
\item $x\in \mathop{dom}(H)$ implies $H(x)'\in \mathop{dom}(G)$. 
\end{enumerate}

Then $(T_{{\mathbf A}}^{\text{DMP}},R_{G})$ is a  frame with $R_G$ 
and $ (R_G)^{-1}$ serial such that 
$({\mathbf A};G,H)$ can be embedded into 
$(\mathbf{M}_2^{T_{{\mathbf A}}^{\text{DMP}}};\widehat{G},\widehat{H})$.
\end{theorem}
\begin{proof} It is enough to check that $R_G=(R_H)^{-1}$.

Let $(s,t)\in R_G$, i.e., 
$(\forall x\in \mathop{dom}(G))(s(G(x))\leq t(x))$. 
We have to check that $(t,s)\in R_H$. Let $y\in \mathop{dom}(H)$. 
Then by assumption (b) we obtain that 
$H(y)'\in \mathop{dom}(G)$. It follows that 
$G(H(y)')'=F(H(y)$ is defined and by axiom (P3) we get 
that $G(H(y)')'\leq  y$, i.e., $y'\leq G(H(y)')$. Since 
$(s,t)\in R_G$ we obtain that 
$s(y')\leq s(G(H(y)'))\leq t(H(y)')$. But $s$ and $t$ are morphisms 
of bounded \dMp{s} which yields that 
$ t(H(y))\leq s(y)$. Hence $R_G\subseteq (R_H)^{-1}$. 
A symmetry argument  gives us that 
$R_H\subseteq (R_G)^{-1}$, i.e., $R_G=(R_H)^{-1}$.
\end{proof}

From Theorem \ref{sercDMPdreprest}  we obtain the following. 

\begin{corollary}\label{totsercDMPdreprest} 
 Let   $({\mathbf A};G,H)$  be  
a dynamic De Morgan algebra.
Then $(T_{{\mathbf A}}^{\text{DMP}},R_{G})$ is a  frame with $R_G$ 
and $ (R_G)^{-1}$ serial such that 
$({\mathbf A};G,H)$ can be embedded into 
$(\mathbf{M}_2^{T_{{\mathbf A}}^{\text{DMP}}};\widehat{G},\widehat{H})$.
\end{corollary}

\begin{remark} Usually, one uses complex algebras associated with the given model 
to establish a discrete duality between algebraic and relational models. Figallo and Pelaitay in \cite{figallo}  
established a discrete duality between tense distributive   De Morgan  algebras and so-called tense De Morgan spaces. 
 Since we are only interested in the representation of tense operators we 
use relational models without any additional structure. 
\end{remark}

\section*{Acknowledgements}  
This is a pre-print of an article published as \newline
 I. Chajda, J. Paseka, Set representation of partial dynamic De Morgan
algebras, In: Proceedings of the 46th IEEE International Symposium on
Multiple-Valued Logic, Springer, (2016), 119--124, doi: 10.1109/ISMVL.2013.56. 
The final authenticated version of the article is available online at: 
\newline 
https://ieeexplore.ieee.org/stamp/stamp.jsp?tp=\&\-arnumber=6524667.

Both authors acknowledge the support by a bilateral project 
New Perspectives on Residuated Posets  financed by  
Austrian Science Fund (FWF): project I 1923-N25, 
and the Czech Science Foundation (GA\v CR): project 15-34697L. 
J.~Paseka acknowledges the financial support of the Czech Science Foundation
(GA\v CR) un\-der the grant  Algebraic, Many-valued and Quantum Structures for Uncertainty Modelling:  
project  15-15286S.

\end{document}